\documentclass[ reqno]{amsart}
\usepackage{graphicx} % Required for inserting images

\title[Vector valued continuous function space]{Vector valued continuous function spaces as $C^\ast$-algebras}
 \author[N. Hotwani]{Neha Hotwani${}^1$}
\address{Department of Mathematics\\
Shiv Nadar Institution of Eminence. Gautam Buddha
Nagar-201314, India}
\email{neha.hotwani@snu.edu.in, nehahotwani19@gmail.com}

\author[T. S. S. R. K. Rao]{T. S. S. R. K. Rao${}^2$}
\address{Department of Mathematics\\
Shiv Nadar Institution of Eminence. Gautam Buddha
Nagar-201314, India}
\email{srin@fulbrightmail.org}


\usepackage{amsmath,amsthm, amsfonts, amssymb, setspace, color, enumerate, bbold}
      \newtheorem{theorem}{Theorem}[section]

      \newtheorem{remark}[theorem]{Remark}
      \newtheorem{lemma}[theorem]{Lemma}
      
      \newtheorem{proposition}[theorem]{Proposition}

      \def\R{{\mathbb R}}
      \def\N{{\mathbb N}}

      \def\cW{\mathcal W}
      
      \def\cA{\mathcal A}

      \def\bb1{\mathbb 1}

\usepackage{ulem}
\usepackage{titlesec}
\usepackage{xstring}

\titleformat{\section}
  {\normalfont\Large\bfseries\centering} % format
  {\thesection}{1em}{}                   % numbering + spacing

\renewcommand{\thesection}{\arabic{section}}
\newcommand{\trace}{\operatorname{trace}}

\usepackage{ulem}
\date{}

\usepackage[dvipsnames]{xcolor}

\usepackage[pagebackref,colorlinks,linkcolor=Maroon,citecolor=MidnightBlue,urlcolor=NavyBlue,hypertexnames=true]{hyperref}

    \newtheorem{corollary}[theorem]{Corollary}

\begin{document}
\begin{abstract}
Let \(\Omega\) be a compact Hausdorff space, and let \(\cA\) be a unital \(C^*\)-algebra. In this study, we continue our examination of the comparison between \(C^*\)-extreme points and linear extremal structures of the unit ball, in the vector-valued \(C^*\)-algebra \(C(\Omega, \cA)\), building upon the work initiated in \cite{HR}. We first enlarge the class of $C^\ast$-algebras in which a $ C^\ast$-extreme point is an extreme point. We demonstrate that if \(\cA\) has a faithful tracial state, then any \(C^*\)-extreme point of the unit ball  \(C(\Omega, \cA)_1\) is a unitary. Additionally, we identify a classes of \(C^*\)-algebras where the concepts of \(C^*\)-extreme and pointwise \(C^*\)-extreme points in \(C(\Omega, \cA)_1\) coincide. We show this holds if a von Neumann algebra has a separable predual with the Radon-Nikodým property.
\end{abstract}
\maketitle

\section{Introduction}
\label{sec:Intro}
Let $\cA$ be a unital $C^\ast$-algebra with identity $\mathbf{1}_\cA$ and let $\cA_1$ denote the closed unit ball of $\cA$. We recall from (\cite{LP}) that
 an element $x \in \cA_1$ is said to be a \textbf{$C^*$-convex combination} of $k$ elements $x_1,\dots, x_k\in \cA_1$, if there exist $t_1, \dots, t_k\in \cA$ such that $\sum_{i=1}^kt_i^*t_i=\bold{1}_\cA$ and $x=\sum_{i=1}^kt_i^*x_it_i$. The $t_i$'s are known as the \textbf{coefficients} of this $C^*$- convex combination. If the coefficients, i.e., the $t_i$'s, are invertible, then this $C^*$-convex combination is called a \textbf{proper $C^*$- convex combination}. 
 $x\in \cA_1$ is said to be a \textbf{$C^*$-extreme point} of $\cA_1$ if, whenever $x$ can be written as a proper $C^*$-convex combination of $x_1, \dots, x_k \in \cA_1$, that is,
\[
x= \sum_{i=1}^kt_i^*x_it_i,
\]
where $t_1,\dots, t_k\in \cA$ are invertible with $\sum_{i=1}^kt_i^*t_i=\bold{1}_\cA$, then each $x_i$ is unitarily equivalent to $x$, i.e., there exist unitaries $u_1, \dots, u_k\in \cA$ such that $x_i=u_i^*xu_i$ for $i=1, \dots, k$.

Let \(\Omega\) be a compact Hausdorff space, and let \(X\) be a Banach space. We denote the space of all \(X\)-valued continuous functions by \(C(\Omega, X)\), which is equipped with the supremum norm. The study of the extremal structure of functions \(f \in C(\Omega, X)_1\) and the pointwise behavior of \(f(\omega) \in X_1\) for \(\omega \in \Omega\) is a well-explored area, as noted in \cite{DHS} and its references.
Recall that $x\in X_1$, is called \textbf{strongly extreme point} if  for any sequences $\{x_n\}$ and $\{y_n\}$ in $X_1$,  $\frac{x_n+y_n}{2} \to x$ implies $x_n-y_n \to 0$. 
% and $y_n \to x$.
For more details, see \cite{DHS}.
 In \cite[Proposition~7]{DHS}, it was shown that \(f \in C(\Omega, X)_1\) is a strongly extreme point if and only if \(f(\omega)\) is a strongly extreme point in \(X_1\) for every \(\omega \in \Omega\). 

In our comparison of algebraic and linear extremal notions, as outlined in \cite{HR}, we demonstrated that in the unit ball of a von Neumann algebra, the concepts of \(C^*\)-extreme points, extreme points, and strongly extreme points coincide. Consequently, we established that in \(\mathcal{A}_1\), any extreme point is also a strongly extreme point.
\vskip 1em

Motivated by the findings in \cite{DHS} and \cite{HR}, we examine the space \( C(\Omega, \mathcal{A}) \), where \( \Omega \) is a compact Hausdorff space and \( \mathcal{A} \) is a unital \( C^* \)-algebra. It is well established that \( C(\Omega, \mathcal{A}) \) is a unital \( C^* \)-algebra. Additionally, if \( \Omega \) is infinite and \( \mathcal{A} \) is an infinite-dimensional von Neumann algebra, then \( C(\Omega, \mathcal{A}) \) cannot be a von Neumann algebra.

We are particularly interested in whether the property that \( f \in C(\Omega, \mathcal{A})_1 \) is a \( C^* \)-extreme point implies that \( f(\omega) \) is a \( C^* \)-extreme point of \( \mathcal{A}_1 \) for every \( \omega \in \Omega \). In \cite{HR}, we were able to address this question when the set of isolated points in \( \Omega \) was dense and \( \mathcal{A} \) was a von Neumann algebra \( \mathcal{M} \). In this article, we extend that result to the case where \( \Omega \) is a compact Hausdorff space and \( \mathcal{A} \) is a $C^\ast$-algebra with a faithful tracial state. Along the way, we also show that in a $C^\ast$-algebra with a faithful tracial state
$C^\ast$-extreme points of the unit ball are precisely unitaries (Corollary \ref{c:ptwise C_ext in faithful C_alg}).

We denote by $C^*\text{-}\operatorname{ext}\bigl(C(\Omega,\cA)\bigr)_1$, the set of $C^\ast$-extreme points and drop the prefix $C^\ast$, to denote the set of linear extreme points. We refer to Kadison's monographs \cite{KR, KR2} for the terminology and results we use here.

 \section{Main results}
 \label{sec:main}
For a given $C^*$-algebra $\cA$, it is easy to see that any unitary is a $C^\ast$-extreme point of $\cA_1$. For instance, see \cite[Remark~2.2]{HR}. Also, $f \in C(\Omega,\cA)_1$ is a unitary if and only if $f(\omega)$ is a unitary in $\cA$ for all $\omega \in \Omega$. This section examines a similar behavior in the $C^\ast$-extreme point context.

We begin with a result showing that $C^\ast$-algebra-valued $C^\ast$-extreme points take values in the unit sphere. We denote by $\mathbf{1}_{\cA}$ the identity of $\cA$ and by $\mathbf{1}$ the constant function on $\Omega$, taking $\mathbf{1}_{\cA}$ as values.
 \begin{theorem}

Let $\Omega$ be a compact Hausdorff space and $\cA$ be a unital $ C^\ast$-algebra. If
$f\in C^*\text{-}\operatorname{ext}\bigl(C(\Omega,\cA)\bigr)_1$,
then
$\|f(\omega)\|=1$
for every $\omega\in\Omega$.
\end{theorem}

\begin{proof}
Suppose, to the contrary, that there exists $\omega_0\in\Omega$ such that
$\|f(\omega_0)\|<1.$
Set
$\delta=1-\|f(\omega_0)\|>0$.
Since $f$ is norm continuous, there exists an open neighborhood $U$ of
$\omega_0$ such that
\[
\|f(\omega)\|\leq 1-\frac{\delta}{2}
\qquad\text{for all }\omega\in U.
\]
Since $\Omega$ is a compact Hausdorff space, as the sets
$\{\omega_0\}$
and
$\Omega\setminus U$
are disjoint by Urysohn's lemma, there
exists
$\varphi\in C(\Omega)$ 
such that $0 \leq \varphi\leq 1$ and 
$\varphi(\omega_0)=1$
and
$\varphi(\omega)=0$ for all $\omega\in\Omega\setminus U$.
Choose
$0<\varepsilon<\frac{\delta}{2}$,
and define
\[
g_{\pm}(\omega)
=
f(\omega)\pm \varepsilon\varphi(\omega)\mathbf{1}_\cA,
\qquad \omega\in\Omega.
\]
Clearly,
$g_{\pm}\in C(\Omega,\cA)$.
We claim that
$g_{\pm}\in C(\Omega,\cA)_1$.
If $\omega\notin U$, then $\varphi(\omega)=0$, and hence
$g_{\pm}(\omega)=f(\omega)$,
so
$\|g_{\pm}(\omega)\|\leq 1$.
On the other hand, if $\omega\in U$, then
\[
\begin{aligned}
\|g_{\pm}(\omega)\|
&\leq \|f(\omega)\|
   +\varepsilon\varphi(\omega)\\
&\leq 1-\frac{\delta}{2}+\varepsilon\\
&<1.
\end{aligned}
\]
Thus
$\|g_{\pm}\|\leq 1$.
Moreover,
$f=\frac{g_++g_-}{2}$.
Equivalently,
\[
f
=
\left(\frac{1}{\sqrt{2}}\mathbf{1}\right)^*
g_+
\left(\frac{1}{\sqrt{2}}\mathbf{1}\right)
+
\left(\frac{1}{\sqrt{2}}\mathbf{1}\right)^*
g_-
\left(\frac{1}{\sqrt{2}}\mathbf{1}\right).
\]
% Since $\frac{1}{\sqrt{2}}\mathbf{1}$ is invertible,
Note that this is a proper
$C^*$-convex decomposition of $f$.
Since $f$ is a $C^*$-extreme point of $C(\Omega,\cA)_1$, there exists a
unitary
$u\in C(\Omega,\cA)$
such that
$g_+=u^*fu$.
Evaluating at $\omega_0$, we obtain
$g_+(\omega_0)
=
u(\omega_0)^*f(\omega_0)u(\omega_0)$.
Hence $g_+(\omega_0)$ and $f(\omega_0)$ are unitarily equivalent and
therefore they have the same spectrum. However since $\varphi(\omega_0)=1$,
$g_+(\omega_0)
=
f(\omega_0)+\varepsilon \mathbf{1}_\cA$.
Therefore,
$\sigma\bigl(g_+(\omega_0)\bigr)
=
\sigma\bigl(f(\omega_0)\bigr)+\varepsilon$, 
which is a contradiction.
Therefore,
$\|f(\omega)\|=1$
for every $\omega\in\Omega$.
\end{proof}
In the case of $\ell^\infty$-direct sum of $C^\ast$-algebras, it was shown in \cite[Theorem~3.6]{HR} that a unit vector $x$ in the direct sum is a $C^\ast$-extreme point if and only if each coordinate is a $C^\ast$-extreme point of the unit ball of the corresponding coordinate algebra. The following corollary is easy to see.
\begin{corollary}
    Let $\cA= \oplus_{n=1}^\infty \cA_n$ be the $\ell^\infty$-direct sum of a sequence $\{\cA_n\}_{n \geq 1}$ of $C^\ast$-algebras. 
    If
$f\in C^*\text{-}\operatorname{ext}\bigl(C(\Omega,\cA)\bigr)_1$,
then
$\|f(\omega)(n)\|=1$
for every $\omega\in\Omega$ and $n \geq 1$.
\end{corollary}

We next exhibit classes of $C^\ast$-algebras where $C^\ast$-extreme points are linear extreme points.
 \begin{theorem}
     \label{t:c_extimp lin ext}
    Let $\mathcal A$ be a unital $C^*$-algebra admitting a faithful tracial
state $\tau$. Let $x\in \cA_1$.
The following are equivalent.
\begin{enumerate} [(i)]
    \item $x$ is unitary.
    \item $x$ is a $C^*$-extreme point of $\cA_1$.
    \item $x$ is a linear extreme point of $\cA_1$.
\end{enumerate}
 \end{theorem}

\begin{proof}  
$(i) \implies (ii)$ follows from \cite[Remark~2.2]{HR}.

$(ii) \implies (iii)$:
    Let $\tau$ be a faithful trace on $\cA$ and $H_\tau$ denote the Hilbert space corresponding to $\tau$ coming from the GNS construction. For $a\in \cA$, let $\Hat{a} \in H_\tau$ denote the corresponding element in $H_\tau$. Observe that the map
  $\psi: \cA \to H_\tau$ given by $\psi(a)= \Hat{a}$
     is a linear map.  The inner product on $H_\tau$ is defined as
    \[
    \left \langle \Hat{a}, \Hat{b} \right \rangle = \tau(ab^*).
    \]
    So, the norm on $H_\tau$ is $\|\Hat{a}\|_\tau=\tau(aa^*)=\tau(a^*a)$.
    Additionally, since $\tau$ is a faithful trace, it follows that the map $\psi$ is injective. 
    % Now, let $x \in \cA_1$ be a $C^*$-extreme point of $\cA_1$. 
    Let 
    \begin{equation}
    \label{eq:conv-comb}
    x= \frac{y+z}{2},
    \end{equation}
    where $y,z \in \cA_1$. We show that $y=z$. Note that Equation \eqref{eq:conv-comb} can also be written as
    \[
    x= \left(\frac{1}{\sqrt{2}}\mathbf{1_\cA}\right)^* y \left(\frac{1}{\sqrt{2}}\mathbf{1_\cA}\right) + \left(\frac{1}{\sqrt{2}}\mathbf{1_\cA}\right)^* z \left(\frac{1}{\sqrt{2}}\mathbf{1_\cA}\right),
    \]
    which denotes a proper $C^*$-convex combination of $x$ in terms of $y$ and $z$. Since $x$ is a $C^*$-extreme point of $\cA_1$, it follows that $y$ and $z$ are unitarily equivalent to $x$. That is, there exist unitaries $u_1, u_2 \in \cA$ such that 
    \[
    y=u_1^*xu_1 \quad \text{and} \quad z= u_2^*xu_2.
    \]
    Therefore, $\tau(y^*y)= \tau(u_1^*x^*xu_1)$ and $\tau(z^*z)= \tau(u_2^*x^*xu_2)$. Thus,
    \[
    \tau(x^*x)= \tau(y^*y)= \tau(z^*z).
    \]
    In other words, $\|\Hat{x}\|_\tau= \|\Hat{y}\|_\tau= \|\Hat{z}\|_\tau$. Hence $\Hat{x}=\Hat{y}= \Hat{z}$. Since $\psi$ is an injective map, it follows that $x=y=z$.
    
$(iii) \implies (i)$:
Let
$(\pi_\tau,H_\tau,\xi_\tau)$
be the GNS representation associated with $\tau$, and put
$\mathcal M=\pi_\tau(\mathcal A)''$.
Since $\tau$ is faithful, $\pi_\tau$ is injective. Indeed, if
$\pi_\tau(a)=0$ for some $a\in\mathcal A$, then
$\tau(a^*a)
=
\|\pi_\tau(a)\xi_\tau\|^2
=
0$,
and the faithfulness of $\tau$ implies that $a=0$.
The tracial state $\tau$ extends to a faithful normal tracial state
$\widetilde{\tau}$ on $\mathcal M$, given by
\[
\widetilde{\tau}(y)
=
\langle y\xi_\tau,\xi_\tau\rangle,
\qquad y\in\mathcal M.
\]
Consequently, $\mathcal M$ is a finite von Neumann algebra. Indeed, if
$v\in\mathcal M$ is an isometry, then $v^*v=\mathbf{1}_H$, and traciality gives
\[
\widetilde{\tau}(vv^*)
=
\widetilde{\tau}(v^*v)
=
\widetilde{\tau}(\mathbf{1}_H)
=
1.
\]
Hence
$\widetilde{\tau}(\mathbf{1}_H-vv^*)=0$.
Since $\widetilde{\tau}$ is faithful, it follows that
$vv^*=\mathbf{1}_H$.
Thus, every isometry in $\mathcal M$ is unitary.
The representation $\pi_\tau$ extends uniquely to a normal surjective
$*$-homomorphism
\[
\widetilde{\pi}_\tau:\mathcal A^{**}\longrightarrow \mathcal M
\]
satisfying
$\widetilde{\pi}_\tau|_{\mathcal A}=\pi_\tau$. Since $x$ is a linear extreme point of $\cA_1$, 
from \cite[Remark~1.5]{HR} it follows that $x$ is a linear extreme point of $(\mathcal A^{**})_1$.
Therefore, 
by \cite[Theorem~10.2]{T}, $x$ is a partial isometry and
\[
(\mathbf{1_\cA}-xx^*)\mathcal A^{**}(\mathbf{1_\cA}-x^*x)=\{0\}.
\]
Let
$v=\pi_\tau(x)=\widetilde{\pi}_\tau(x)\in\mathcal M$.
Since $\widetilde{\pi}_\tau$ is a $*$-homomorphism, $v$ is a partial
isometry. We claim that
\[
(\mathbf{1}_H-vv^*)\mathcal M(\mathbf{1}_H-v^*v)=\{0\}.
\]
Indeed, let $y\in\mathcal M$. Since $\widetilde{\pi}_\tau$ is surjective,
there exists $z\in\mathcal A^{**}$ such that
$\widetilde{\pi}_\tau(z)=y$.
Therefore,
\[
\begin{aligned}
(\mathbf{1}_H-vv^*)y(\mathbf{1}_H-v^*v)
&=
\widetilde{\pi}_\tau(\mathbf{1_\cA}-xx^*)\,
\widetilde{\pi}_\tau(z)\,
\widetilde{\pi}_\tau(\mathbf{1_\cA}-x^*x)\\
&=
\widetilde{\pi}_\tau
\bigl((\mathbf{1_\cA}-xx^*)z(\mathbf{1_\cA}-x^*x)\bigr)\\
&=0.
\end{aligned}
\]
Hence
$(\mathbf{1}_H-vv^*)\mathcal M(\mathbf{1}_H-v^*v)=\{0\}$.
By \cite[Theorem~10.2]{T}, $v$ is a linear extreme point of
$\mathcal M_1$.
Since $\mathcal M$ is a finite von Neumann algebra, it follows from \cite[Theorem~1.2, Theorem~3.12]{HR} that every linear extreme
point of $\mathcal M_1$ is unitary. Therefore,
$v^*v=vv^*=\mathbf{1}_H$.
Thus
\[
\pi_\tau(x^*x)
=
\pi_\tau(x)^*\pi_\tau(x)
=
\mathbf{1}_H
=
\pi_\tau(\mathbf{1_\cA}),
\]
and hence
$\pi_\tau(\mathbf{1_\cA}-x^*x)=0$.
Since $\pi_\tau$ is injective on $\mathcal A$, it follows that
$x^*x=\mathbf{1_\cA}$.
Similarly,
$xx^*=\mathbf{1_\cA}$.
Therefore, $x$ is unitary in $\mathcal A$.  This completes the proof.
% Finally, from \cite[Remark~2.2]{HR} 
% $x\in C^*\text{-}\operatorname{ext}(\mathcal A_1)$.
% Let
% $(\pi_\tau,H_\tau,\xi_\tau)$
% be the GNS representation associated with $\tau$.
% Since $\tau$ is faithful, $\pi_\tau$ is injective. Indeed, if
% $\pi_\tau(a)=0$ for some $a\in\mathcal A$, then
% $\tau(a^*a)
% =
% \|\pi_\tau(a)\xi_\tau\|^2
% =
% 0$,
% and the faithfulness of $\tau$ implies that $a=0$.
% \noindent
\end{proof}

The next corollary deals with the vector-valued case of Theorem \ref{t:c_extimp lin ext}.
\begin{corollary}
    \label{c:ptwise C_ext in faithful C_alg}
    Let $\cA$ be a unital $C^*$-algebra such that $\cA$ has a faithful tracial state $\tau$. Let $f\in C^*\text{-}\operatorname{ext}\bigl(C(\Omega,\cA)\bigr)_1$. Then $f$ is a unitary.
\end{corollary}

\begin{proof}
    First, we show that $f$ is a linear extreme point of $C(\Omega, \cA)_1$. For that, let 
    \[
    f=\frac{g_1+g_2}{2},
    \]
    where $g_1, g_2 \in C(\Omega, \cA)_1$. We have,
    \[
    f=\left(\frac{1}{\sqrt{2}}\mathbf{1}\right)^* g_1 \left(\frac{1}{\sqrt{2}}\mathbf{1}\right) + \left(\frac{1}{\sqrt{2}}\mathbf{1}\right)^* g_2 \left(\frac{1}{\sqrt{2}}\mathbf{1}\right),
    \]
    which is a proper $C^*$-combination of $f$ in terms of $g_1$ and $g_2$. Since $f$ is a $C^*$-extreme point of $C(\Omega, \cA)_1$, it follows that $g_1$ and $g_2$ are unitarily equivalent to $f$. That is, there exist unitaries $u_1$ and $u_2$ in $C(\Omega, \cA)$ such that $g_i=u_i^*fu_i$ for $i=1,2$. Therefore, for each $\omega \in \Omega$ and $i=1,2$, we have $g_i(\omega)=u_i(\omega)^*f(\omega)u_i(\omega)$. Since $\cA$ has a faithful tracial state $\tau$,  by applying the same procedure as in the case $(ii) \implies (iii)$ in the proof of Theorem \ref{t:c_extimp lin ext}, one obtains for each $\omega \in \Omega$
    \[
    g_1(\omega)= f(\omega)= g_2(\omega).
    \]
    In other words, $g_1=f=g_2$. Hence $f$ is a linear extreme point of $C(\Omega, \cA)_1$. 
    Now, using \cite[Theorem~B]{HR}, it follows that $f$ is a strongly extreme point of $C(\Omega, \cA)_1$. By appealing \cite{DHS}, we have $f(\omega)$ is a strongly extreme point of $\cA_1$ for all $\omega \in \Omega$. By Theorem \ref{t:c_extimp lin ext}, $f(\omega)$ is a unitary for all $\omega \in \Omega$. Hence $f$ is a unitary.
\end{proof}
 
Recall that the space $M_n$ is a Hilbert space with respect to the Hilbert-Schmidt norm, which is given by $\|A\|_2= \sqrt{\trace(A^*A)}$ for $A\in M_n$. Using the property that for any $A, B\in M_n$, we have $\trace(AB)= \trace(BA)$, one can see that $\|A\|_2= \|U^*AU\|_2$, where $U\in M_n$ is a unitary matrix. We use this observation in the following theorem.

\begin{remark}
\label{rem:unitary in direct sum}
We note the coordinate-wise computation in an $\ell^\infty$-direct sum for future use. 
Let
\[
\mathcal A
=
\bigoplus_{k=1}^{\infty} M_k
=
\left\{
(A_k)_{k=1}^{\infty} :
A_k\in M_k
\text{ for all } k,\ 
\sup_{k\ge 1}\|A_k\|<\infty
\right\}.
\]
Let \(A=(A_k), B=(B_k)\in\mathcal A\), where \(A_k,B_k\in M_k\).
Suppose that
$B=U^*AU$
for some unitary \(U=(U_k)\in\mathcal A\). Since multiplication and
adjoints in \(\mathcal A\) are defined coordinatewise, the identities
$U^*U=UU^*=\mathbf{1}_\cA$
imply that
$U_k^*U_k=U_kU_k^*=\mathbf{1}_k$
for every \(k\), where $\mathbf{1}_k$ denotes the identity of $M_k$. Thus, each \(U_k\) is a unitary in \(M_k\). Moreover,
from \(B=U^*AU\), we obtain
$B_k=U_k^*A_kU_k$, for all $k\ge 1$.
Therefore, \(B_k\) is unitarily equivalent to \(A_k\) for every \(k\). 

Similar computations also work for partial isometries.
\end{remark}

 We give an independent proof in the case of matrices.
 \begin{theorem}
     \label{t:Mn case}
     Let $\cA= \oplus_{k=1}^\infty M_k$.
      If $f \in C(\Omega, \cA)_1$ is a $C^*$-extreme point of $C(\Omega, \cA)_1$ then $f$ is a linear extreme point of $C(\Omega, \cA)_1$.
 \end{theorem}

 \begin{proof} We follow the procedure outlined during the proof of Corollary \ref{c:ptwise C_ext in faithful C_alg}.
     Let 
     \[
     f=\frac{g_1+g_2}{2},
     \]
      where $g_1, g_2 \in C(\Omega, \cA)_1$.
     Now
     \[
     f=\left(\frac{1}{\sqrt{2}}\mathbf{1}\right)^* g_1 \left(\frac{1}{\sqrt{2}}\mathbf{1}\right) + \left(\frac{1}{\sqrt{2}}\mathbf{1}\right)^* g_2 \left(\frac{1}{\sqrt{2}}\mathbf{1}\right).
     \]
     % where $\mathbf{1}$ denote the identity element of $C(\Omega, M_n)_1$. 
     Since $f$ is a $C^*$-extreme point of $C(\Omega, M_n)_1$, there exist unitaries $u_1, u_2 \in C(\Omega, M_n)$ such that 
     \[
     g_1=u_1^*fu_1 \quad g_2=u_2^*fu_2.
      \]
      Therefore for each $i=1,2$, one has
      \[
      g_i(\omega)= u_i(\omega)^*f(\omega)u_i(\omega) \quad \text{for all}~~ \omega \in \Omega.
      \]
   Let $k\in \N$ and  $\omega \in \Omega$ be arbitrary fixed.  Now using Remark \ref{rem:unitary in direct sum}, we have 
      \[
      g_i(\omega)_k= u_i(\omega)_k^*f(\omega)_ku_i(\omega)_k.
      \]
      In particular, for each $i=1,2$, one obtains $\|g_i(\omega)_k\|_2= \|f(\omega)_k\|_2$.
      Now, applying the parallelogram law, we get
      \begin{align*}
\|g_1(\omega)_k-g_2(\omega)_k\|_2^2=& 2\|g_1(\omega)_k\|_2^2+ \|g_2(\omega)_k\|_2^2-\|g_1(\omega)_k+g_2(\omega)_k\|_2^2\\
 =& 2 \|f(\omega)_k\|_2^2+ 2 \|f(\omega)_k\|_2^2 -\|2 f(\omega)_k\|_2^2\\
 =&0.
     \end{align*}
     Thus, we have  $g_1(\omega)_k=g_2(\omega)_k$. Since $k \in \N$ and $\omega \in \Omega$ were arbitrarily fixed, it follows that $g_1=g_2$.
     This completes the proof.
 \end{proof}

 \begin{corollary}
     \label{c:C_ext-iff-ptwise}
     Let $\cA= \oplus_{k=1}^\infty M_k$ and let $f\in C(\Omega, \cA)_1$ be a $C^\ast$-extreme point. Then $f$ is a unitary,
 \end{corollary}

 \begin{proof}
     Since $f$ is a $C^*$-extreme point of $C(\Omega, \cA)_1$  by Theorem \ref{t:Mn case}, we have $f$ is a linear extreme point of $C(\Omega, \cA)_1$. Using again the fact that $C(\Omega, \cA)$ is a unital $C^*$-algebra and applying \cite[Theorem~B]{HR}, it follows that $f$ is a strongly extreme point of $C(\Omega, \cA)_1$. Now from \cite{DHS}, it follows that $f(\omega)$ is a strongly extreme point of $\cA$ for all $\omega \in \Omega$.  It is easy to see that components of $f(\omega)$ are extreme points of the corresponding $M_k$'s, and hence are unitaries, so that $f(w)$ is a unitary for all $\omega \in \Omega$. Therefore, $f$ is unitary.
 \end{proof}
We next explore in the vector-valued case the relationship between $C^*$-extreme and pointwise partial isometries.

 \begin{lemma}
 \label{lem:pointwise-partial-isometry}
Let $\Omega$ be a compact Hausdorff space and let $H$ be a separable
Hilbert space. Suppose that
$f\in C^*\text{-}\operatorname{ext}\bigl(C(\Omega,B(H))_1\bigr)$.
Then $f(\omega)$ is a partial isometry for every $\omega\in\Omega$.
\end{lemma}

\begin{proof}
Fix $\omega_0\in\Omega$ and put
$x=f(\omega_0)$, $T=x^*x$.
Since $\|x\|\leq1$, we have
$0\leq T\leq \mathbf{1}_H$.
Suppose, to the contrary, that $T$ is not a projection. Then
\[
\sigma(T)\cap(0,1)\neq\varnothing.
\]
Indeed, if $\sigma(T)\subseteq\{0,1\}$, then the continuous functional
calculus would give $T^2=T$.
Choose
$\lambda_0\in \sigma(T)\cap(0,1)$,
and let $E_T$ denote the spectral measure of $T$. Since $\lambda_0$
belongs to the spectrum of $T$, one has
$E_T(V)\neq0$
for every open neighbourhood $V$ of $\lambda_0$.
Consider the continuous function $F: (0,1) \times (0,1) \to \R$ defined as
\[
F(s,c)
=
\frac{1}{\sqrt{s}}-1
-
\left|\sqrt{\frac{c}{s}}-1\right|.
% \qquad (s,c)\in(0,1)\times(0,1).
\]
Since
$F(\lambda_0,\lambda_0)
=
\frac{1}{\sqrt{\lambda_0}}-1>0$,
there exists an open interval $U=(a,b)$ containing $\lambda_0$, with
$\overline U\subset(0,1)$,
such that for $s,c \in U$
\begin{equation}
\label{eq:r-bound}
\left|\sqrt{\frac{c}{s}}-1\right|
<
\frac{1}{\sqrt{s}}-1.
%\qquad (s,c\in U).
\end{equation}
Since $E_T(U)\neq0$, choose a compact interval $J\subset U$ such that
$E_T(J)\neq0$.
% For instance, writing
% \[
% J_n=
% \left[
% a+\frac{b-a}{n+2},
% b-\frac{b-a}{n+2}
% \right],
% \]
% we have $J_n\uparrow U$, and hence
% \[
% E_T(J_n)\xrightarrow{\mathrm{SOT}}E_T(U).
% \]
% Therefore $E_T(J_n)\neq0$ for some $n$.
Now, $T$ is self-adjoint, and $H$ is separable; $T$ has at most
countably many distinct eigenvalues. Since $U$ is uncountable, we may
choose
$c\in U$
which is not an eigenvalue of $T$.
Choose $\phi\in C([0,1],[0,1])$ such that
\[
\phi|_J=1
\qquad\text{and}\qquad
\phi=0\quad\text{on }[0,1]\setminus U.
\]
Define
\[
r(s)=
\begin{cases}
\displaystyle
\phi(s)\left(\sqrt{\frac{c}{s}}-1\right),
&s\in U,\\[1.2ex]
0,&s\notin U.
\end{cases}
\]
Then $r\in C([0,1],\mathbb R)$ and, by \eqref{eq:r-bound},
\[
|r(s)|
\leq
\frac{1}{\sqrt{s}}-1
\qquad(0<s\leq1).
\]
Consequently,
$|1\pm r(s)|
\leq
\frac{1}{\sqrt{s}}$,
and hence
\begin{equation}\label{eq:scalar-contraction-lemma}
s(1\pm r(s))^2\leq1,
\qquad 0\leq s\leq1.
\end{equation}
For $\omega\in\Omega$, let
$T_\omega=f(\omega)^*f(\omega)$,
and define
\[
g_\pm(\omega)
=
f(\omega)\bigl(I\pm r(T_\omega)\bigr).
\]
By continuous functional calculus,
$g_\pm\in C(\Omega,B(H))$.
Moreover,
$\frac{g_++g_-}{2}=f$.
By \eqref{eq:scalar-contraction-lemma},
\[
\begin{aligned}
g_\pm(\omega)^*g_\pm(\omega)
&=
T_\omega\bigl(\mathbf{1}_H\pm r(T_\omega)\bigr)^2\\
&\leq \mathbf{1}_H,
\end{aligned}
\]
and therefore
$g_\pm\in C(\Omega,B(H))_1$.
Thus
\[
f
=
\left(\frac{\mathbf{1}}{\sqrt2}\right)^*
g_+
\left(\frac{\mathbf{1}}{\sqrt2}\right)
+
\left(\frac{\mathbf{1}}{\sqrt2}\right)^*
g_-
\left(\frac{\mathbf{1}}{\sqrt2}\right)
\]
is a proper $C^*$-convex decomposition of $f$.
Since $f$ is $C^*$-extreme, $g_+$ is unitarily equivalent to $f$.
Hence, there exists a unitary
$u\in C(\Omega,B(H))$
such that
$g_+=u^*fu$.
Evaluating at $\omega_0$ gives
$g_+(\omega_0)
=
u(\omega_0)^*f(\omega_0)u(\omega_0)$,
and therefore
\begin{equation}\label{eq:positive-unitary-equivalence}
g_+(\omega_0)^*g_+(\omega_0)
=
u(\omega_0)^*Tu(\omega_0).
\end{equation}
Define
$\psi(s)=s(1+r(s))^2$.
Then
$g_+(\omega_0)^*g_+(\omega_0)=\psi(T)$.
Since $\phi=1$ on $J$, for $s\in J$ we have
$r(s)=\sqrt{\frac{c}{s}}-1$,
and hence
\[
\psi(s)
=
s\left(\sqrt{\frac{c}{s}}\right)^2
=c.
\]
Therefore,
\[
\psi(T)E_T(J)=cE_T(J).
\]
Since $E_T(J)\neq0$, $c$ is an eigenvalue of $\psi(T)$.
On the other hand, $c$ was chosen not to be an eigenvalue of $T$.
Hence $\psi(T)$ cannot be unitarily equivalent to $T$, contradicting
\eqref{eq:positive-unitary-equivalence}.
Therefore, $T=x^*x$ is a projection, and hence $x=f(\omega_0)$ is a
partial isometry. Since $\omega_0$ was arbitrary, $f(\omega)$ is a
partial isometry for every $\omega\in\Omega$.
\end{proof}
We do not know whether Lemma \ref{lem:pointwise-partial-isometry} holds for any Hilbert space. 
\begin{theorem}\label{thm:pointwise-cstar-extreme}
Let $\Omega$ be a compact Hausdorff space 
and 
% let $H$ be a separable
% Hilbert space. If
$f\in C^*\text{-}\operatorname{ext}\bigl(C(\Omega,B(H))_1\bigr)$.
Suppose $f(\omega)$ be a partial isometry for all $\omega \in \Omega$.
Then
\[
f(\omega)\in C^*\text{-}\operatorname{ext}(B(H)_1)
\qquad\text{for every }\omega\in\Omega.
\]
Consequently,
$f\in C(\Omega,B(H))_1)$
is a strongly extreme point.
\end{theorem}

\begin{proof}
% By Lemma~\ref{lem:pointwise-partial-isometry}, $f(\omega)$ is a partial
% isometry for every $\omega\in\Omega$. 
Define
\[
p(\omega)=\mathbf{1}_H-f(\omega)^*f(\omega),
\qquad
q(\omega)=\mathbf{1}_H-f(\omega)f(\omega)^*.
\]
Then $p$ and $q$ are norm continuous projection-valued functions.
We claim that
$p(\omega)=0$
\text{or}
$q(\omega)=0$
for every $\omega\in\Omega$.
Suppose, to the contrary, that for some $\omega_0\in\Omega$,
$p(\omega_0)\neq0$
\text{and}
$q(\omega_0)\neq0$.
Write
$p_0=p(\omega_0)$,
$q_0=q(\omega_0)$.
Choose unit vectors
\[
\xi\in p_0H,
\qquad
\eta\in q_0H,
\]
and define the rank-one partial isometry $k\in B(H)$ by
$k\zeta=\langle\zeta,\xi\rangle\eta$.
Then
$q_0kp_0=k$ and
$k^*k$ is a rank one projection.
% =P_{\mathbb C\xi}.
Fix $0<\varepsilon<1$ and define
\[
h(\omega)
=
\varepsilon q(\omega)kp(\omega),
\qquad\omega\in\Omega,
\]
and
$F_\pm=f\pm h$.
Since $p$ and $q$ are norm continuous, one has
$h,F_\pm\in C(\Omega,B(H))$.
Since $f(\omega)$ is a partial isometry, we obtain
$f^*q=0$ \text{and}
$qf=0$.
Thus
$f^*h=h^*f=0$.
Moreover,
\[
\begin{aligned}
h^*h
&=
\varepsilon^2pk^*qkp\\
&\leq
\varepsilon^2pk^*kp\\
&\leq
\varepsilon^2p.
\end{aligned}
\]
It follows that
\[
\begin{aligned}
F_\pm^*F_\pm
&=
f^*f+h^*h\\
&\leq
I-p+\varepsilon^2p\\
&\leq I.
\end{aligned}
\]
Hence
$F_\pm\in C(\Omega,B(H))_1$,
and
$f=\frac{F_++F_-}{2}$
is a proper $C^*$-convex decomposition.
Since $f$ is $C^*$-extreme, $F_+$ must be unitarily equivalent to $f$.
However, at $\omega_0$,
$h(\omega_0)=\varepsilon k$,
so
\[
F_+(\omega_0)^*F_+(\omega_0)
=
\mathbf{1}-p_0+\varepsilon^2k^*k.
\]
Since $p_0\xi=\xi$ and $k^*k\xi=\xi$,
\[
F_+(\omega_0)^*F_+(\omega_0)\xi
=
\varepsilon^2\xi.
\]
Thus $\varepsilon^2\in(0,1)$ is an eigenvalue of
$F_+(\omega_0)^*F_+(\omega_0)$. Hence, this operator is not a projection,
and therefore $F_+(\omega_0)$ is not a partial isometry.
But $f(\omega_0)$ is a partial isometry. Thus $F_+(\omega_0)$ cannot be
unitarily equivalent to $f(\omega_0)$, a contradiction.
Therefore,
\[
p(\omega)=0
\qquad\text{or}\qquad
q(\omega)=0
\]
for every $\omega\in\Omega$. Equivalently,
\[
f(\omega)^*f(\omega)=\mathbf{1}_H
\qquad\text{or}\qquad
f(\omega)f(\omega)^*=\mathbf{1}_H.
\]
Thus $f(\omega)$ is either an isometry or a coisometry.
By 
\cite[Theorem~1.1 and Corollary~1.2]{HMP}, the $C^*$-extreme points
of $B(H)_1$ are precisely the isometries and coisometries. Hence
$f(\omega)\in C^*\text{-}\operatorname{ext}(B(H)_1)$
 for all $\omega\in\Omega$.
Finally, by \cite[Theorem~3.2]{HR}, isometries and coisometries are strongly extreme points of
$B(H)_1$.
\end{proof}

\begin{corollary}
    \label{cor:H sep}
    Let $H$ be a separable Hilbert space and 
$f\in C^*\text{-}\operatorname{ext}\bigl(C(\Omega,B(H))_1\bigr)$. Then
\[
f(\omega)\in C^*\text{-}\operatorname{ext}(B(H)_1)
\qquad\text{for every }\omega\in\Omega.
\]
\end{corollary}

\begin{proof}
    It follows immediately from Lemma \ref{lem:pointwise-partial-isometry} and Theorem \ref{thm:pointwise-cstar-extreme}.
\end{proof}

 In the proof of the next corollary, we use a result of Chu involving spaces with the Radon-Nikodým property (RNP), see \cite{C}. We refer \cite{DU} for more details.
 For a Hilbert space $H$, let ${\mathcal K}(H)$ denote the space of compact operators.
 \begin{corollary}
     Let $\cW^\ast$ be a von Neumann algebra such that the predual has the RNP and is separable. If
\[
f\in C^*\text{-}\operatorname{ext}\bigl(C(\Omega, \cW^\ast\bigr),
\]
then
\[
f(\omega)\in C^*\text{-}\operatorname{ext}(\cW^\ast_1)
\qquad\text{for every }\omega\in\Omega.
\]
 \end{corollary}
 \begin{proof}
     Since $\cW$ has the RNP, it follows from Chu that there is a family $\{H_{\alpha}\}_{\alpha \in \Delta}$ of Hilbert spaces such that $\cW$ is isometric to the $\ell^1$-direct sum of the spaces, ${\mathcal K}(H_{\alpha})^\ast$. Since $\cW$ is also separable, we get that the index set $\Delta$ is countable, and the Hilbert spaces are separable. Thus $\cW^\ast$ is isometrically $C^\ast$-isomorphic to  $\oplus_{k=1}^\infty {\mathcal B}(H_k)$. Now, the conclusion follows by applying the proof of the above corollary coordinate-wise.
 \end{proof}
 
 In the case of a von Neumann algebra, in \cite[ Corollary~3.11]{HR}, we showed that any vector in the unit ball is an average of two $C^\ast$ extreme points. In particular, they generate the von Neumann algebra. We do not know whether a similar result holds for vector-valued functions taking values in a $C^\ast$-algebra. See the recent article \cite{AC} for more information on spans and convex combinations of boundary-valued functions. 
 
 For a discrete set $\Gamma$, let $\beta(\Gamma)$ denote the Stone-\v{C}ech compactification.
 \begin{proposition}
     Let $f \in C(\beta(\Gamma), M_n)_1$ .Then $f = \frac{g+h}{2}$ , where $g,h$ are unitaries.
 \end{proposition}
\begin{proof}
    We know that for any $\omega \in \Gamma$, $f(\omega)$ is an average of two unitaries. Thus we have functions $g,h$ on $\Gamma$ with unitaries as values and such that $f = \frac{g+h}{2}$ on $\Gamma$. By the compactness of the unit ball, we can extend $g, ~ h$ as continuous functions to $\beta(\Gamma)$ denoted by $g',~h' \in C(\beta(\Gamma), M_n)_1$, such that $f = \frac{g'+h'}{2}$. Clearly  since $\Gamma$ is dense in $\beta(\Gamma)$, we see that $g'(\omega),~h'(\omega)$ are unitaries for $\omega \in \beta(\Gamma)$. The conclusion follows.
\end{proof}

\vspace{.3cm}
 \noindent
\textbf{Acknowledgements:} This work is part of the project ``Classification of Banach spaces using differentiability",
funded by the Anusandhan National Research Foundation (ANRF), Core Research Grant, CRG2023-000595.
The first author is a research associate (RA) in this project. The authors thank Mr. Chinmay Ajay Tamhankar (IITM) for his comments and
suggestions on Theorem~2.3.

% \vspace{.3cm}
% \textbf{Declaration:}
% No conflicts of interest.

\bibliographystyle{alpha}
\bibliography{references}
\end{document}